\documentclass[11pt,a4paper]{article}
\usepackage{graphicx}
\usepackage{amsfonts}
\usepackage{amssymb}
\usepackage{amsthm}
\usepackage[centertags]{amsmath}
\usepackage{newlfont}

\usepackage{color}
\usepackage{graphicx,color}
\usepackage{amsmath, amssymb, graphics}

\def\r{\mathbb R}

\def\n{\mathbf n}

\def\N{\mathbf N}

\def\F{\mathbf F}
\def\rr{\mathbf r}

\newtheorem{theorem}{Theorem}[section]
\newtheorem{corollary}[theorem]{Corollary}
\newtheorem{proposition}[theorem]{Proposition}
\newtheorem{definition}[theorem]{Definition}

\newtheorem{remark}[theorem]{Remark}

\title{Constant-speed ramps for a central force field}

 \author{Rafael L\'opez\footnote{Partially supported by the grant no. MTM2017-89677-P, MINECO/AEI/FEDER, UE}\\
 Departamento de Geometr\'{\i}a y Topolog\'{\i}a\\ Instituto de Matem\'aticas (IEMath-GR)\\
 Universidad de Granada\\
 18071 Granada, Spain\\
\texttt{rcamino@ugr.es}
\and
\'Oscar Perdomo\\
Department of Mathematics\\
 Central Connecticut State University\\
New Britain, CT 06050, USA\\
\texttt{perdomoosm@ccsu.edu}}

\date{}

\begin{document}\maketitle

\begin{abstract} 

We investigate the problem of determining the planar curves that describe ramps where a particle of mass $m$  moves with constant-speed when is subject to the action of the friction force and a force whose magnitude $F(r)$ depends only on the distance $r$ from the origin.    In this paper we  describe all the constant-speed ramps for the  case  $F(r)=-m/r$. We show the circles and the logarithmic spirals play an important role. No only they are solutions but every other solution approaches either a circle or a logarithmic spiral. 
\end{abstract}


\section{Formulation of the problem} 

It is known that when an object  is placed on an inclined plane with small angle $\theta$ with the horizontal, the friction forces are responsible that the object stays still on the plane.  
 As we tilt the plane and  due to the gravity, there is a critical angle $\theta_0$, the so-called  angle of repose, such that  the object starts to slide down the ramp. If, in addition,  we tilt more the ramp,  the object accelerates down the inclined plane. 
  The force of gravity acting on the object is separated in two components. While the normal component to the plane is balanced with the normal force that exerts the plane, 
  the   component  parallel to the plane is bigger than the friction force that opposes the motion of the object. These unbalanced forces make that the object slides down with constant acceleration. At the critical slope $\theta_0$ of the plane,   the net force on the object is zero and the object slides down on the inclined plane with constant-speed $v$. This allows to measure the coefficient of kinetic friction $\mu$ between the object and the floor of the ramp, deducing that $\tan\theta_0=\mu$, in particular, $\mu$ is independent of $v$. The  history behind the deduction of the angle of repose is long and goes back to works by, among others,   Stevin,    Galilei,   da Vinci and    Euler: see \cite[Ch. V]{us} for a historical introduction.

Recently, the second author considered the   problem of determining the   non-rectilinear curve contained in a vertical plane that describes an object that moves down with constant-speed $v$  by the effect of the gravity and the  friction forces (\cite{pe}). Besides the tilted straight line of angle $\theta_0$ described in the above paragraph,  it was proven that constant speed ramps can also be built using a rotation of the trajectory   $\alpha=\alpha(t)$   parametrized by    
$$\alpha(t)=\left(t+\frac{1}{a}\log(1+e^{-2at}),\frac{2}{a} \mathrm{arccot}(e^{-at})\right),$$
where $a$ is a constant depending on $\mu$,  $v$ and the acceleration due to gravity. In general, given a set of forces and a coefficient of
friction $\mu$, the problem of constant-speed ramp consists of finding a curve that describes the boundary of ramp where a particle will move with constant speed.

Before we continue, we precise the definition of a ramp because we have to say which is the side about $\alpha$ where the object is  supported. This will be indicated when we fix a normal vector to $\alpha$. By convention, the unit normal vector  $\mathbf{n}$  to a regular curve $\alpha$ is defined by rotating the unit tangent vector $\alpha'(t)/|\alpha'(t)|$ counterclockwise through an angle $\pi/2$, 
$$\mathbf{n}(t)= J\left(\frac{\alpha'(t)}{|\alpha'(t)|}\right), \quad  J(x,y)=(-y,x).$$
  Locally, the trace of $\alpha$ separates the plane $\r^2$ in two components whose common boundary is $\alpha$ and only one of them has the vector $\mathbf{n}$ as the outer normal vector. In other words, given  $t_0$ in the domain of $\alpha$, we require that the   set  
\begin{equation}\label{ss}
\Sigma_{\alpha}(t_0)=\{\alpha(t)-u  \mathbf{n}(t) : t\in (t_0-\delta,t_0+\delta), 0<u<\epsilon\}
\end{equation}
for $\delta,\epsilon>0$ sufficiently small, is part of the ramp. Formally, the definition of a ramp is the following.

\begin{definition} \label{d-1}
A {\it ramp} is an ordered pair $(\alpha,\mathbf{n})$, where   $\alpha:[a,b]\rightarrow\r^2$ is a regular curve and $\mathbf{n}$ is its unit normal vector field. The ramp  $(\alpha,\mathbf{n})$  will be viewed locally as the domain $\Sigma_{\alpha}(t_0)$   described in \eqref{ss}.    
 \end{definition}

In this paper we study the constant-speed ramp problem under the effect of a central force field. A {\it central force} is defined  as   a force   that points from the particle directly towards, or away,  from  a fixed point in the plane called the center, and whose magnitude depends only on  the distance of the object to the center (\cite{go}). We precise the formulation of our problem.

 Consider the motion of a particle $M$ of  constant mass $m$ in  the plane $xOy$ of the Cartesian rectangular coordinate system
along a smooth curve  $\alpha=\alpha(t)$.   From now on, we use symbols in a
bold font to denote vectors and we use the same non-bold symbols  to represent
their magnitudes. We assume that the forces exerted on $M$ are the following: see Figure \ref{fig1}.

\begin{itemize}
\item A {\it central force} $\mathbf{F}$ that  acts on the particle $M$ directed  towards, or away,  from a fixed point in plane, which we assume to be   the origin $O$ of the coordinate system. The magnitude of the central force will depend only on the   the distance $r$ between $O$ and the particle $M$. We also assume that $\F\not=0$ everywhere. If $\mathbf{r}$ is the vector position of $M$,  the central force writes as 
$\mathbf{F}(\mathbf{r})=F(r)\hat{\rr}$  where $\hat{\rr}$ is the unit vector in the direction of $\mathbf{r}$. Here $F(r)$ is positive (repulsion) or negative (attraction).
\item A {\it normal force} $\mathbf{N}$ that is  exerted on  $M$ by the ramp $\alpha$    orthogonal to the movement of $M$.  Thus $N$ can be expressed as 
$$\mathbf{N}(t)=\lambda(t) \mathbf{n}(t),$$
where $\lambda>0$.
\item A friction force $\mathbf{F}_f$  in the opposite direction of the movement. The force $\mathbf{F}_f$ can be expressed as 
$$\mathbf{F}_f=-\mu\lambda(t)\alpha'(t),$$
 where $\mu>0$ is a constant called the kinetic coefficient of friction.  
 \end{itemize}

Under this system of forces, the purpose of this paper is the following.

\begin{quote}{\bf Problem:}  Determine the  ramps $\alpha=\alpha(t)$  that induce a motion of $M$ with {\it constant speed} under the effect of a central force field and the friction force. 
\end{quote}

\begin{figure}[hbtp]
\centering \scalebox{0.8}{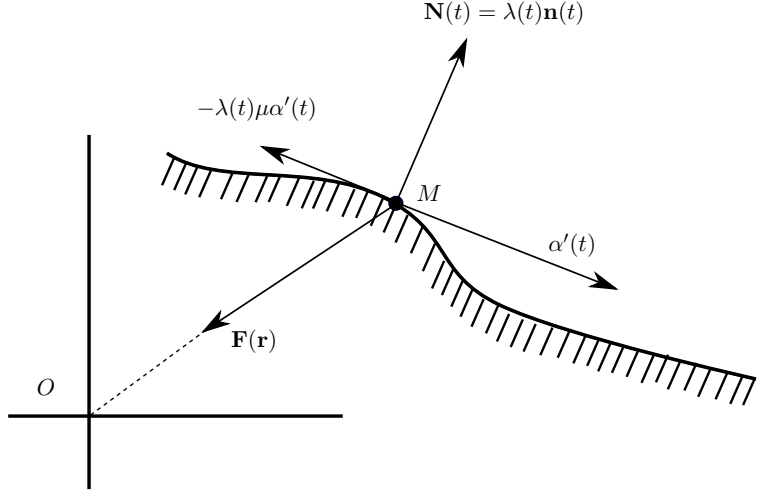} 
\caption{A mass $M$ sliding along $\alpha$ under the effect of a central force $\F(\rr)$ and the friction force. }\label{fig1}
\end{figure}

An example of central force is when the force is directly/inversely proportional to the $n$th power of the distance $r$. In this paper, we call the {\it $n$th power central force} to the force defined by 
\begin{equation}\label{pcf}
\F(\rr)=\varepsilon m r^n \hat{\rr},
\end{equation}
where $n\in\r$ and $\varepsilon\in\{-1,1\}$. The value $\varepsilon=1$ (resp. $\varepsilon=-1$) indicates an attraction (resp. repulsion) force.  

In Section \ref{sec2} we obtain the kinematic equation of the motion of $M$ for a general central force field $\F(\rr)$ when we impose the restriction that the speed of $M$ is constant. We point out that this equation is difficult to solve in all its generality, even in the case of example \eqref{pcf}. In that section we also  study   what types of central force fields have   constant-speed ramps with simple geometry, as for example, straight-lines or circles. 

Due to the difficulty of the general problem, in this paper we   focus in the particular case  
\begin{equation}\label{pcf1}
\F(\rr)=-m\frac{\hat{\rr}}{r}.
\end{equation}

We will call this force an {\it inverse central harmonic oscillator} due to the fact that the force for the regular harmonic oscillator is given 
by   $\F(\mathbf{r})=-r \hat{\rr}$.   In Sections \ref{sec3} and \ref{sec4}, we will investigate the geometric description of the constant-speed ramps for the choice \eqref{pcf1}. Among the results, we point out the following observations.  
\begin{enumerate}
\item[(a)] In contrast with the case of the inclined plane,   the shapes of our ramps depend on the velocity $v$.
\item[(b)] If $v=1$, there are circles as examples of constant-speed ramps for every value of the radius. On the other hand, if $v\not=1$, there do not exist circular constant-speed ramps.
\item[(c)] If  $v=1$, all  constant-speed ramps are  circles or curves bounded between two concentric circles and having these circles as limit points. (Theorem \ref{tv1}).
\item[(d)] If $v\not=1$, there are two constant-speed ramps $\alpha_{ls}$ that are logarithmic spirals. Besides these examples, the other ramps converge to $\alpha_{ls}$ (Corollary \ref{c-43}). If $v>1$ the ramps are unbounded whereas if $v<1$, they are bounded. 
\item[(e)] The only periodic constant-speed ramps are circles centered at the origin.
 \end{enumerate}
 
 For the description of the constant-speed ramps, we will use the notion of     TreadmillSled introduce by the second author in \cite{pe1}. The TreadmillSled operator associates to each plane curve $\alpha$ other plane curve $\gamma$ constructed as follows. Let us take two copies of the Euclidean plane, one of them, say $\r^2$, is fixed and the other one, denoted by $P$, can move freely with respect to $\r^2$. On $P$, we place the curve $\alpha$. Then the TreadmillSled $\gamma$ of $\alpha$ is the geometric locus that traces the origin $O\in P$ in the fixed plane $\r^2$  when we move $P$ in such a way that we are placing every $\alpha(t)$ (which is in plane $P$)   on top of the origin   $(0,0)$ making sure that $\alpha^\prime(t)$ points toward the positive $x$ direction. This definition is independent of re-parametrizations that preserve  orientation.  In our study, firstly we will find the TreadmillSled of constant-speed ramps and then we do an inverse process to recover the initial curve.   
 
\section{Equation of the constant-speed ramp}\label{sec2}

In this section we derive the  equations of motion for $M$ moving with constant-speed on a ramp under  the action of a central force field and a friction force: see Figure \ref{fig1}. 
Since the particle slides on the ramp, both, the trajectory of the particle and the shape of the ramp are described with the same curve, which we parametrize by arc-length $s$. According with the purpose of this paper, we will assume that $M$ moves with constant speed $v\not=0$. We reparametrize $\alpha$ to have constant-speed $v$ by means of  $\beta(t)=\alpha(vt)$.  The friction force is $\F_f(s)=-\mu\lambda(s)\alpha'(s)$, where $\mu>0$ is the kinematic friction coefficient and the normal force to the ramp  is $\N=\lambda\n$.  Recall that the vector $\n$ determines physically the ramp by Definition \ref{d-1}. Finally the central force $\F$ is 
 $$\F(\rr)=F(r)\hat{\rr}=\frac{F(r)}{r}\beta(t),\quad r=|\beta(t)|.$$

By  Newton's second law,  the differential equation describing the particle motion is 
 \begin{equation}\label{eq1}
 m \beta''(t)=\F(\rr)+\N(t)+\F_f(t).
 \end{equation}
The acceleration of $\beta$ is 
$$\beta''(t)=v^2\alpha''(s)=v^2\kappa(s)\n(s),$$
where $s=vt$ and $\kappa$ is the curvature of $\alpha$ defined as $\kappa(s)=\langle\alpha''(s),\mathbf{n}(s)\rangle$. Thus \eqref{eq1} is now
\begin{equation}\label{eq2}
 mv^2\kappa(s)\n(s)=\frac{F(r)}{|\alpha(s)|} \alpha(s)+\lambda(s)\n(s)-\mu\lambda(s)\alpha'(s).
 \end{equation}
By  using the Euclidean scalar product $\langle ,\rangle$,  we multiply    \eqref{eq2} by $\alpha'(s)$ and $\n(s)$, obtaining respectively
 $$0=\frac{F(r)}{|\alpha(s)|}\langle\alpha(s),\alpha'(s)\rangle-\mu\lambda(s)$$
$$ mv^2\kappa(s)=\frac{F(r)}{|\alpha(s)|}\langle\alpha(s),\n(s)\rangle +\lambda(s).$$
From both expressions we obtain $\lambda$, and  we deduce the characterization of a constant-speed ramp under a central force field.

\begin{theorem}\label{t1}
 Let $M$ be a particle of mass $m$ and let  $\F(\rr)$ be a central force field. Then  $\alpha=\alpha(s)$ is   constant-speed ramp  with velocity $v$ under the effect of $\F$  if and only if $\frac{F(r)}{|\alpha(s)|}\langle\alpha(s),\alpha'(s)\rangle\ge0$ and the curvature   $\kappa$  satisfies
\begin{equation}\label{mm}
\kappa(s)=  \frac{F(r)}{mv^2 r}\left(\frac{\langle\alpha(s),\alpha'(s)\rangle}{\mu}+  \langle\alpha(s),\n(s)\rangle\right),
\end{equation}
where $\mu$ is the constant friction coefficient and $r=|\alpha(s)|$. Here $s$ is the arc-length parameter of $\alpha$.

If $\F(\rr)$ is  a $n$th power central force, then  
\begin{equation}\label{ff2}
\kappa(s)= \frac{\varepsilon}{v^2} r^{n-1}\left( \frac{ \langle\alpha(s),\alpha'(s)\rangle}{\mu}+ \langle\alpha(s),\n(s)\rangle\right).
\end{equation}
\end{theorem}

  \begin{remark}
If a curve $\alpha(s)$ satisfies Equation \eqref{mm} but not the condition 
$ \frac{F(r)}{|\alpha(s)|}\langle \alpha(s),\alpha^\prime(s) \rangle\ge0$, then $\tilde{\alpha}=\alpha(-s)$ satisfies both condition to be a ramp. For this reason, every curve that satisfies Equation \eqref{mm} defines a ramp if its orientation (the selection of the solid part of the ramp) is selected properly.\end{remark}

In view of \eqref{eq2}, the sign of $\varepsilon$ in \eqref{pcf} has the following interpretation.  We know that 
 $$\langle\alpha(s),\alpha'(s)\rangle=\frac{\mu\lambda(s)}{F(r)} |\alpha(s)|=\varepsilon\frac{\mu\lambda(s) }{m} r^{1-n},$$
 hence
 $$\mbox{sgn}\left(\frac{d}{ds}|\alpha(s)|^2\right)=\mbox{sgn}(\varepsilon).$$
Thus if $\varepsilon=-1$ (resp. $\varepsilon=1$), the function $s\mapsto |\alpha(s)|$ is decreasing (resp. increasing). This implies   that at every point $s$ of the curve, the unit tangent of $\alpha$ points towards (resp. away) the round disc centered at $O$ of radius $|\alpha(s)|$. In other words, the trajectory of $\alpha$ goes `inside' (resp. `outside').

We particularize the  arguments of the proof of Theorem \ref{t1}  in case of no friction forces ($\mu=0$). Recall that if the particle moves  on a ramp with not friction with constant speed under the effect of the gravity force $(0,-mg)$, then, due to the conservation of the total energy, kinetic energy plus potential energy, this ramp must be a horizontal line. In case of a central force, we prove the following result.

\begin{proposition} Without friction, the constant-speed ramps   under the action of a central force field are circles centered at the origin.
\end{proposition}

\begin{proof} 

If $\mu=0$, we deduce   from \eqref{eq2} that $\langle\alpha(s),\alpha'(s)\rangle=0$ for all $s$.  Thus the function $|\alpha(s)|^2$ is constant and this shows that the trace of $\alpha$ is contained in a circle centered at the origin $O$.  
\end{proof}

We study the existence of constant-speed ramps with simple geometries,    for example,    straight-lines and circles. 

\begin{corollary} For a central force field, there do not exist linear constant-speed ramps.
\end{corollary}

\begin{proof}
The proof is by contradiction. Suppose that   a straight-line $\alpha=\alpha(s)$ is  a solution of \eqref{mm}. If $\alpha$ parametrizes as $\alpha(s)=p_0+s\mathbf{u}$, $p_0\in\r^2$, $|\mathbf{u}|=1$, then $\kappa=0$ and \eqref{mm} implies
$$\langle\alpha(s),\alpha'(s)\rangle+\mu\langle\alpha(s),\n(s)\rangle=0.$$
Since $\alpha'(s)=\mathbf{u}$ and $\n(s)=J\mathbf{u}$,  we deduce
$$s+\langle p_0,\mathbf{u}\rangle+\mu\langle p_0,J\mathbf{u}\rangle=0$$
for all $s\in\r$, which is impossible. This contradiction proves the result.
\end{proof}

We study now the    constant-speed ramps that are circles. In the following result, we will prove the existence of circular constant-speed ramps, where the radius of the circle is not arbitrary, but it depends on the speed $v$.

 \begin{corollary} For a central force field $\F(\rr)$, a circle is a constant-speed ramp if and only if its centre is the origin $O$  and its radius $R$ satisfies 
 $$F(R)=-\frac{mv^2}{R}.$$
\end{corollary}

\begin{proof} Suppose that the circle parametrizes as
$$\alpha(s)=(a,b)+R\left(\cos(s/R),\sin(s/R)\right).$$
Then $\kappa=1/R$ and 
$$\alpha'(s)=\left(-\sin(s/R),\cos(s/R)\right),\ \ \n(s)=-\left(\cos(s/R),\sin(s/R)\right).$$
Replacing $\alpha$ and $\mathbf{n}$ in Equation  \eqref{mm}, we obtain 
\begin{equation}\label{FF}
 \frac{mv^2 r}{R}=F(r)\left(\frac{-a\sin(s/R)+b\cos(s/R)}{\mu} -(a\cos(s/R)+b\sin(s/R))-R\right),
 \end{equation}
 where 
 $$r^2=|\alpha(s)|^2=a^2+b^2+R^2+2R(a\cos(s/r)+b\sin(s/R)).$$
Squaring \eqref{FF} and writing in one hand side, we obtain  
$$A_0+A_1\cos(s/R)+A_2\sin(s/R)+A_3\cos(2s/R)+A_4\sin(2s/R)=0,$$
where $A_i$ are constants independent from $s$. Thus all coefficients $A_i$ vanish identically. The computation of $A_4$ and $A_3$ gives
$$A_4=\frac{F(r)^2}{\mu ^2}(a \mu -b) (a+b \mu ),$$
$$A_3=\frac{F(r)^2}{2 \mu ^2}((a-b)\mu-a-b)((a+b)\mu+a-b).$$
 Since $F\not=0$,   from $A_4=0$ we deduce that $b=a\mu$ or $a=-b\mu$. Suppose $b=a\mu$. Then $A_3=0$ simplifies into
$$-\frac{F(r)^2}{2 \mu ^2} \left(\mu ^2+1\right)^2a^2=0,$$
obtaining $a=0$, hence $b=0$. Similarly, if $a=b\mu$, the equation $A_3=0$ reduces in
$$\frac{b^2 \left(\mu ^2+1\right)^2}{2 \mu ^2}=0,$$
obtaining $b=a=0$ again. In both cases, we have proved that the center $(a,b)$ of the circle is the origin $O$. On the other hand,   \eqref{FF} simplifies into $mv^2=-F(R)R$, proving the result.
\end{proof} 

It is immediate the following consequence.
 
\begin{corollary} \label{c-23}
 For the $n$th power central force  $\F(\rr)=\varepsilon m r^n \hat{\rr}$,  the existence of circular constant-speed ramps occurs when $\varepsilon=-1$ and $v =R^{(n+1)/2}$. Consequently, 
\begin{enumerate}
\item If $n\not=-1$, for each value of $v$, there is   only one  circular ramp. 
\item If $n=-1$ (inverse central harmonic oscillator), the circular ramps only appear if the speed is $v=1$, and in such a case,  any  circle centered at the origin is a constant-speed ramp. 
\end{enumerate}
\end{corollary}

\begin{remark}  In Corollary \ref{c-23} we have found closed periodic constant-speed ramps. This can be viewed in connection with the 
Bertrand's theorem that asserts that the only attractive central potentials   in Euclidean space that can yield closed bounded orbits are the 
harmonic oscillator ($n=1$) and the Newtonian potential ($n=-2$) in \eqref{pcf}:  \cite{be}, see also \cite[Ap. A]{go}.
\end{remark}

Following with the study of constant-speed ramps, it is immediate that if we rotate a solution $\alpha$ of \eqref{mm} with respect to the origin $O$, then the resulting curve is also a solution of \eqref{mm}. We study how   a dilation affects  the shape of  a constant-speed ramp, focusing in the case of the $n$th power central force.

 \begin{corollary}\label{c-dilation} Let $\alpha=\alpha(s)$ be a  constant-speed ramp for  the $n$th power central force $\F(\rr)$  and velocity $v$. If $c >0$, then $\eta(s)=c \alpha(s/c )$ is a constant-speed ramp for $\F(\rr)$   with the same   friction constant   and velocity $c ^{(n+1)/2}v$.  In particular, for $n=-1$, the velocity does not change.
\end{corollary}

\begin{proof} Since $\eta$ is parametrized by arc-length, $\n_\eta=J\eta'(s)=J\alpha'(s)=\n(s)$. Hence, its curvature $\kappa_\eta$ is 
$$\kappa_\eta(s)=\langle\eta''(s),\n_\eta(s)\rangle=\frac{1}{c}\langle\alpha''(s/c),\n_\alpha(s/c)\rangle=\frac{\kappa_\alpha(s)}{c}.$$
  If we write \eqref{ff2} in terms of $\eta$, we obtain
$$c \kappa_\eta(s)=\varepsilon  \frac{|\eta(s)|^{n-1}}{v^2c ^{n-1}} \left( \frac{1}{\mu}  \langle\dfrac{\eta(s)}{c },\eta'(s)\rangle + \langle\frac{\eta(s)}{c },\n_\eta(s)\rangle\right),$$
which simplifies into
$$\kappa_\eta(s)=\varepsilon \frac{|\eta(s)|^{n-1}}{c^{n+1}v^2  } \left( \frac{ \langle\eta(s) ,\eta'(s)\rangle}{\mu}+ \langle\eta(s) ,\n_\eta(s)\rangle\right).$$
 This proves the result.
 \end{proof}

\section{The  inverse central harmonic oscillator: case $v=1$} \label{sec3}

In the next two sections,  we focus in the particular case that the force $\F(\rr)$ is the inverse central harmonic oscillator
\begin{equation}\label{ic}
\F(\rr)=-m\dfrac{\hat{\rr}}{r}.
\end{equation} 
First examples of constant-speed ramps are those described by circles centered at the origin. This was proved in  Corollary \ref{c-23} and we now recall  here again.

 \begin{proposition}\label{pr-31}
  For the inverse central harmonic oscillator $\F(\rr)=-m \hat{\rr}/r$, any circle centered at the origin is  a constant-speed ramp if    $v=1$, and  there are not circular constant-speed ramps if $v\not=1$.
 \end{proposition}
 
Let $\alpha=\alpha(s)$ be a constant-speed ramp parametrized by arc-length $s$. If $\alpha(s)=(x(s),y(s))$,  $s\in I\subset\r$, then  
$$\alpha'(s)=(x'(s),y'(s)),\quad \n(s)=J\alpha'(s)=(-y'(s),x'(s)).$$
Equation \eqref{ff2} is
\begin{equation}\label{eq3}
 \kappa=-\frac{1}{v^2(x^2+y^2)}\left(\frac{xx'+yy'}{\mu}-xy'+yx'\right).
\end{equation}
 We now use the notion of TreadmillSled introduced  in \cite{pe1}. Instead of solving for the curve $\alpha$ by means of  \eqref{eq3},    we first compute the TreadmillSled of $\alpha$ and subsequently, the curve $\alpha$.  Since in this case,   $\alpha$ is parametrized by arc-length, the TreadmillSled of $\alpha$ is   
$$ \gamma(s)=-(\langle\alpha(s),\alpha'(s)\rangle,\langle\alpha(s),\n(s)\rangle):=(\xi_1(s),\xi_2(s)).$$
 By the parametrization of $\alpha$, 
\begin{equation}\label{pg}
 \gamma(s)=-\left(x(s)x'(s)+y(s)y'(s),-x(s)y'(s)+y(s)x'(s)\right).
 \end{equation}
With this notation, Equation \eqref{eq3} is now
\begin{equation}\label{kk}
\kappa=\frac{1}{v^2\mu(x^2+y^2)}\left( \xi_1 +\mu \xi_2\right).
\end{equation}
By using the Frenet equations 
$$\alpha''(s)=\kappa(s)\n(s),\quad \n'(s)=-\kappa(s)\alpha'(s),$$
 we have
\begin{equation}\label{eq0}
\begin{split}&\xi_1'(s)=-1-\langle\alpha(s),\alpha''(s)\rangle=-1+\kappa(s)\xi_2(s)\\
&\xi_2'(s)=-\langle\alpha(s),\n'(s)\rangle=-\kappa(s)\xi_1(s).
\end{split}
\end{equation}
 In these expressions for $\xi_1'$ and $\xi_2'$,  replacing $\kappa$ by \eqref{kk},   and taking into account that 
$$x(s)^2+y(s)^2=|\alpha(s)|^2=|\gamma(s)|^2=\xi_1(s)^2+\xi_2(s)^2,$$
give us    the following autonomous system
\begin{equation}\label{eqs}
\begin{split}
&\xi_1'=-1+\frac{1}{v^2\mu(\xi_1^2+\xi_2^2)}\left( \xi_1 +\mu \xi_2\right)\xi_2\\
&\xi_2' =-\frac{1}{v^2\mu(\xi_1^2+\xi_2^2)}\left( \xi_1 +\mu \xi_2\right)\xi_1.
\end{split}
\end{equation}
We compute the critical points of the above  system. By $\xi_2'=0$, we have   $\xi_1+\mu\xi_2=0$ or $\xi_1=0$. In the first case, from   $\xi_1'=0$, we have $\mu v^2(1+\mu^2)\xi_2^2=0$, that is, $(\xi_1,\xi_2)=(0,0)$, which is not possible. We then conclude  $\xi_1=0$. By $\xi_1'=0$ again,    we deduce   $(v^2-1)\mu\xi_2^2=0$. Since $\xi_2\not=0$, then necessarily $v=1$. 

We have proved that there are infinite critical points, namely, $\{(0,a):a\not=0\}$ which only occur when $v=1$. Each one of these points $(0,a)$ corresponds with a constant TreadmillSled $\gamma(s)=(0,a)$ for all $s\in I$.  From \eqref{pg}, 
$\langle\alpha(s),\alpha'(s)\rangle=0$ and $\langle\alpha(s),\n(s)\rangle=a$. Consequently, the curve  $\alpha$ is a circle centered at the origin $O$, proving   that the critical points of \eqref{eqs} correspond with the (circular)  solutions of Proposition \ref{pr-31}. 

On the other hand, because the vertical line $\xi_1=0$ is formed by stationary points of \eqref{eqs} when $v=1$, any trajectory cannot meet the $\xi_2$-axis. In particular, every trajectory of \eqref{eqs} lies in the half-plane $\xi_1<0$ or in the half-plane $\xi_1>0$. 

In what  follows, we need to know how to recover the curve $\alpha$ in terms of its   TreadmillSled curve   $\gamma$. The proof of the   reversing process  is given in \cite[Prop. 2.11]{P}.

\begin{proposition}\label{pr-t}
 Let $\gamma(t)=(\xi_1(t),\xi_2(t))$ be a regular curve. Then $\gamma$ is the TreadmillSled of a regular curve $\alpha$ if and only if $\xi_2'(t)=-f(t)\xi_1(t)$ for some continuous function $f$ and $\xi_2f-\xi_1'>0$. Moreover, if $G(t)$ is an antiderivative of $f(t)$, then 
\begin{equation}\label{tsa}
\alpha(t)=- \left(\begin{array}{ll}\cos(G(t))&-\sin(G(t))\\ \sin(G(t))&\cos(G(t))\end{array}\right)\gamma(t).
\end{equation}
In particular, $\gamma$ intersects orthogonally the $y$-line. The solution $\alpha$ is unique up to a rotation.
\end{proposition}

We give some  examples   of Proposition \eqref{pr-t}.
\begin{enumerate}
\item For any $r>0$, the point $(0,r)$ is the TreadmillSled of the circle $\alpha(t)=r(\cos t,\sin t)$. This is a direct computation. 

\item Let $\gamma(t)=(h(t),ah(t))$, $a\in\r$, any non vertical half-line that starts at the origin, where $h'<0$. Then $f(t)=-ah'(t)/h(t)$, so $G(t)=-a\log(h(t))$. With the change of variable $h(t)\rightarrow e^{-t}$, then \eqref{tsa} yields
$$\alpha(t)=-e^{-t}\left(\cos(at)-a\sin(at),\sin(at)+a\cos(at)\right).$$
This curve is a logarithmic spiral.
\end{enumerate}
 
 The study of the constant-speed ramps for the inverse central harmonic oscillator will be separated in two cases depending on the value of the velocity $v$, namely, $v=1$ or $v\not=1$.  The first case, that is, $v=1$, will be taken care in this section, whereas the case $v\not=1$ will  be addressed  in Section \ref{sec4}.

Suppose   $v=1$. The system \eqref{eqs} is 
\begin{equation}\label{eqs2}
\begin{split}
&\xi_1'=\frac{\xi_1}{\mu(\xi_1^2+\xi_2^2)}(-\mu\xi_1+\xi_2)\\
&\xi_2' =-\frac{\xi_1}{\mu(\xi_1^2+\xi_2^2)}(\xi_1+\mu\xi_2).
\end{split}
\end{equation}
Let us observe that the computation of $\alpha$ by means of its TreadmillSled $\gamma$ is invariant under reparametrizations of $\gamma$. 
Thus  we can replace  \eqref{eqs2} by the following system  
\begin{equation}\label{eqs3}
\begin{split}
&\phi_1'=-\mu\phi_1+\phi_2\\
&\phi_2' =-\phi_1-\mu\phi_2.
\end{split}
\end{equation}
The only stationary  point of this system is the origin $(0,0)$ and the eigenvalues of the linearized sytem at $(0,0)$ are $-\mu\pm i$. Since $\mu>0$, the complex conjugate numbers have a negative real part and the stationary point   is an stable focus. In  the phase portrait, the  trajectories are spirals that  twist approaching the origin: see Figure \ref{f-phase1}.

\begin{figure}[hbtp]
\begin{center}\includegraphics[width=.5\textwidth]{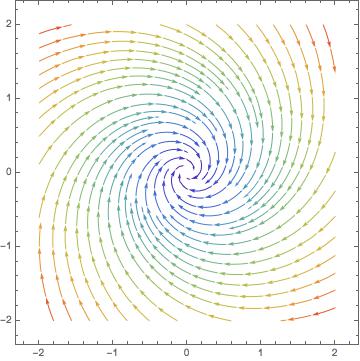}
\end{center}
\caption{Case $v=1$. The phase portrait of the system \eqref{eqs3}, with $\mu=0.5$. The  origin $(0,0)$ is the only equilibrium point and it is a stable focus. }\label{f-phase1}
\end{figure}
 
The solutions of the linear system \eqref{eqs3} are
\begin{equation}\label{eqs4}
\begin{split}
&\phi_1(t)=e^{-\mu t} \left(c_1 \cos (t)+c_2 \sin (t)\right)\\
&\phi_2(t) =e^{-\mu t} \left(c_2 \cos (t)-c_1 \sin (t)\right),
\end{split}
\end{equation}
where $c_1,c_2\in\r$. After a rotation about $O$ and a dilation (Corollary \ref{c-dilation}), we take the initial conditions for the ramp \eqref{ff2}
$$\alpha(0)=(x(0),y(0))=(1,0)$$
$$ \alpha'(0)=(x'(0),y'(0))=(\cos(u), \sin(u)),$$
where $u\in [\pi/2,3\pi/2)$. Then 
$$\xi_1(0)=-\cos(u),\ \xi_2(0)=\sin(u).$$
Notice that if $u=\pi/2$, then  $\gamma(0)=(0,1)$ and, as pointed out before, this point on its own is the TreadmillSled of a circle centered at the origin.  

Assume now that $u\not=\pi/2$. The solution   \eqref{eqs4} is
\begin{equation}\label{eqs5}
\gamma(t)=e^{-\mu t}(- \cos (t+u),  \sin (t+u)).
\end{equation}
 Following the notation of  Proposition \ref{pr-t} to calculate the curve $\alpha$, we have
 $$f(t)=-\frac{\phi_2'(t)}{\phi_1(t)}=1-\mu\tan(t+u)$$
 with the condition 
 $$0<\phi_2(t)f(t)-\phi_1'(t)=-\frac{\mu}{\cos(t+u)}e^{-\mu t}.$$
This implies $\cos(u+t)<0$. Due to the fact that $t$ is near zero and $u$ lies between $\pi/2$ and $3\pi/2$, we deduce that   $t\in (\pi/2-u,3\pi/2-u)$. 
  
 The function $G$ is 
 $$G(t)=\int f(t) dt=t+\mu\log(-\cos(t+u))+k,$$
 where $k\in\r$.    In order to fulfill the initial conditions $x(0)=1$, $y(0)=0$, we take $k=u-\mu\log(-\cos(u))$.   Finally, by \eqref{tsa}
   
 \begin{eqnarray}\label{alphaforveq1}
 \alpha(t)&=&-e^{-\mu t}  \left(\begin{array}{ll}\cos(G(t))&-\sin(G(t))\\ \nonumber
 \sin(G(t))&\cos(G(t))\end{array}\right)\left(\begin{array}{l} - \cos (t+u)\\  \sin (t+u)\end{array}\right)\\ 
 &=&e^{-\mu t}\left(\cos \left(\mu  \log \frac{\cos (t+u)}{\cos(u)}\right),-\sin \left(\mu  \log \frac{\cos (t+u)}{\cos(u)}\right)\right)
 \end{eqnarray}
for $t\in (\pi/2-u,3\pi/2-u)$. 

We summarize the above  arguments in the following result.

\begin{theorem} \label{tv1}
If $v=1$, the TreadsmillSled of the constant-speed ramps for  the inverse central harmonic oscillator are points of the $\xi_2$-axis or (part of) logarithmic spirals. When the TreadmillSled is a logarithmic spiral, the constant-speed ramp is given by the expression \eqref{alphaforveq1}. 
\end{theorem}

 If $v=1$, then the    TreadmillSled curve   parametrized by \eqref{eqs5} is only defined in an interval of $\r$ with the condition that  the trace of $\gamma$ is contained either in the half-plane $\xi_1>0$ or $\xi_1<0$. Thus the whole spiral $\gamma(t)$ cannot be the TreadmillSled of a curve. However,  the connected pieces that lie on the half-planes $\xi_1>0$ or $\xi_1<0$ are ThreadmillSled of curves.  See Figure \ref{fig-v1}. 

\begin{figure}[hbtp]
\begin{center}\includegraphics[width=.4\textwidth]{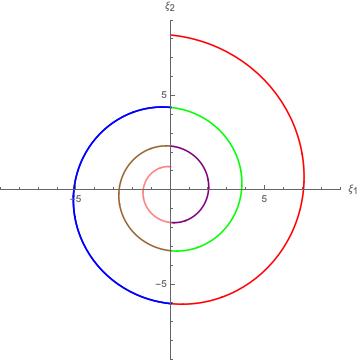}\quad \includegraphics[width=.4\textwidth]{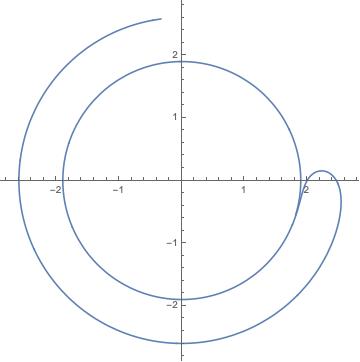}
\end{center}
\caption{Case $v=1$. The purple part of the logarithmic spiral (left) is the TreadmillSled of the non-circular ramp (right).} \label{fig-v1}
\end{figure}

\section{The  inverse central harmonic oscillator: case $v\not=1$} \label{sec4}

We now consider the case $v\not=1$ for the inverse central harmonic oscillator \eqref{ic}. The  TreadmillSled of the constant-speed ramps are given by \eqref{eqs}, which can be expressed as 
\begin{eqnarray*}
&&\xi_1'=\frac{1}{v^2\mu(\xi_1^2+\xi_2^2)}\left(-\mu v^2\xi_1^2+\mu(1-v^2)\xi_2^2+\xi_1\xi_2\right)\\
&&\xi_2' =\frac{1}{v^2\mu(\xi_1^2+\xi_2^2)}\left(- \xi_1^2-\mu \xi_1\xi_2\right).
\end{eqnarray*}
 Since we are interested in the trajectories of the solutions $(\xi_1,\xi_2)$, we will study the following system that share the same trajectories:
\begin{equation}\label{eqs6}
\begin{split}
&\phi_1'=-\mu v^2\phi_1^2+\mu(1-v^2)\phi_2^2+\phi_1\phi_2\\
&\phi_2' =-\phi_1^2-\mu\phi_1\phi_2.
\end{split}
\end{equation}

The only equilibrium point of the above system   is $(0,0)$, which is a degenerate point.

\begin{figure}[hbtp]
\begin{center}\includegraphics[width=.4\textwidth]{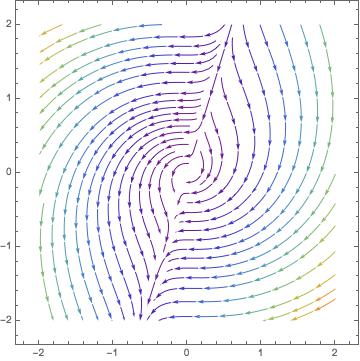}\quad \includegraphics[width=.4\textwidth]{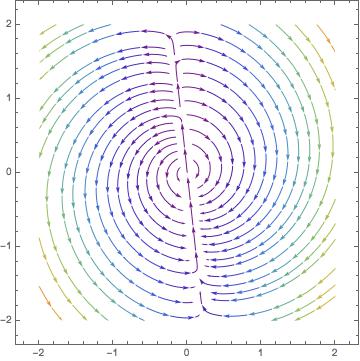}
\end{center}
\caption{The phase portrait of the system \eqref{eqs6}. Left: $v=2$ and $\mu=0.1$. Right: $v=0.5$ and $\mu=0.3$. }\label{f-phase2}
\end{figure}

We begin the study of the quadratic system \eqref{eqs6} by looking for  the solutions   that are straight-lines. If the  solution is    of the form 
$$\gamma(t)=(\phi_1(t),\phi_2(t))=g(t)(a_1,a_2)$$
with $\mathbf{a}=(a_1,a_2)$ a nonzero vector, then \eqref{eqs6} reduces to
\begin{equation}\label{eqs63}
\begin{split}
g'(t)a_1&= g(t)^2(-\mu v^2 a_1^2+\mu(1-v^2)a_2^2+a_1a_2)\\
g'(t)a_2& =g(t)^2(-a_1^2-\mu a_1a_2).
\end{split}
\end{equation}
Since we are only interested in the direction of the semi-line trajectories, we can assume  without loss of generality  that   $g'(t)=g(t)^2$ and 
\begin{eqnarray*}
 a_1&=&  -\mu v^2 a_1^2+\mu(1-v^2)a_2^2+a_1a_2\\
 a_2& =&-a_1^2-\mu a_1a_2.
\end{eqnarray*}
We obtain   $g(t)=1/(\lambda-t)$, $\lambda\in\r$. It   follows that  
$$ \mathbf{a}=\frac{1}{\mu  v^2}\left(-1,  \frac{1}{\mu(1-v^2)}\right).$$
Let us denote
$$r_0= \frac{1}{\mu(1-v^2)}.$$
The ThreadmillSled curve $\gamma$ is 
$$\gamma(t)=\frac{1}{\mu v^2(\lambda-t)}\left(-1, r_0\right).$$
We observe that as $t\rightarrow\infty$, $\gamma(t)$ goes to the origin in the half-line determined by the  direction $-\mathbf{a}$ and $\gamma(t)$ goes away from the origin in the half-line of direction $\mathbf{a}$. See Figure \ref{f-line}. Compare this figure with Figure \ref{f-phase2} where the two straight-lines appear in the phase portrait. We can see how all the non semi-lines trajectories in Figure \ref{f-phase2}  are asymptotic to the semi-lines. We will prove this affirmation once we find a closed formula for all the trajectories.

  \begin{figure}[hbtp]
\begin{center}\includegraphics[width=.4\textwidth]{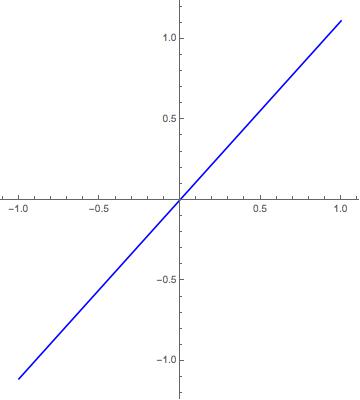}\quad \includegraphics[width=.12\textwidth]{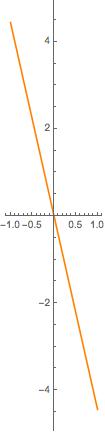}
\end{center}
\caption{  TreadmillSleds that are half-lines.  Left:  $v=2$, $\mu=0.1$. Right: $v=0.5$, $\mu=0.3$. }\label{f-line}
\end{figure}

Since the   inverse process by Proposition \ref{pr-t} is independent of reparametrizations,   let
$$\gamma(t)=(\xi_1(t),\xi_2(t))=\pm\frac{e^t}{\mu v^2}\left(-1,  r_0\right).$$
We now follows the steps given in   Proposition \ref{pr-t}. The function $f(t)$ is  
$$f(t)=-\frac{\xi_2'(t)}{\xi_1(t)}=r_0,$$
where we require the positivity of the function $\xi_2(t)f(t)-\xi_1'(t)$. Now 
$$\xi_2(t)f(t)-\xi_1'(t)=\frac{r_0 e^t}{\mu  v^2 },$$
  which is always positive. The function $G(t)$ is 
  $$G(t)=\int f(t) dt=r_0 t.$$
   By the formula \eqref{tsa},  the expression of $\alpha$ is 
\begin{equation}\label{v2}
\begin{split}
\alpha(t)&=\mp\frac{e^t}{\mu v^2} \left(\begin{array}{ll}\cos(r_0 t )&-\sin(r_0 t )\\ \sin(r_0 t)&\cos(r_0 t)\end{array}\right)\left(\begin{array}{l}-1\\ r_0\end{array}\right)\\
&=\pm\frac{e^t}{\mu v^2} \left(\begin{array}{l} \cos(r_0 t)+r_0\sin(r_0 t)\\
\sin(r_0 t)-r_0\cos(r_0 t)\end{array}
\right)
   \end{split}
  \end{equation}

We can easily see that $\alpha$ is a logarithmic spiral. 
  
  In Figures \ref{f-v2} (resp. Figure \ref{f-v3}) we depict  the constant-speed ramps for the case $v>1$ (resp. $v<1$) for the choices of the vector $ \mathbf{a}$ and $- \mathbf{a}$.
  
  \begin{figure}[hbtp]
\begin{center}\includegraphics[width=.4\textwidth]{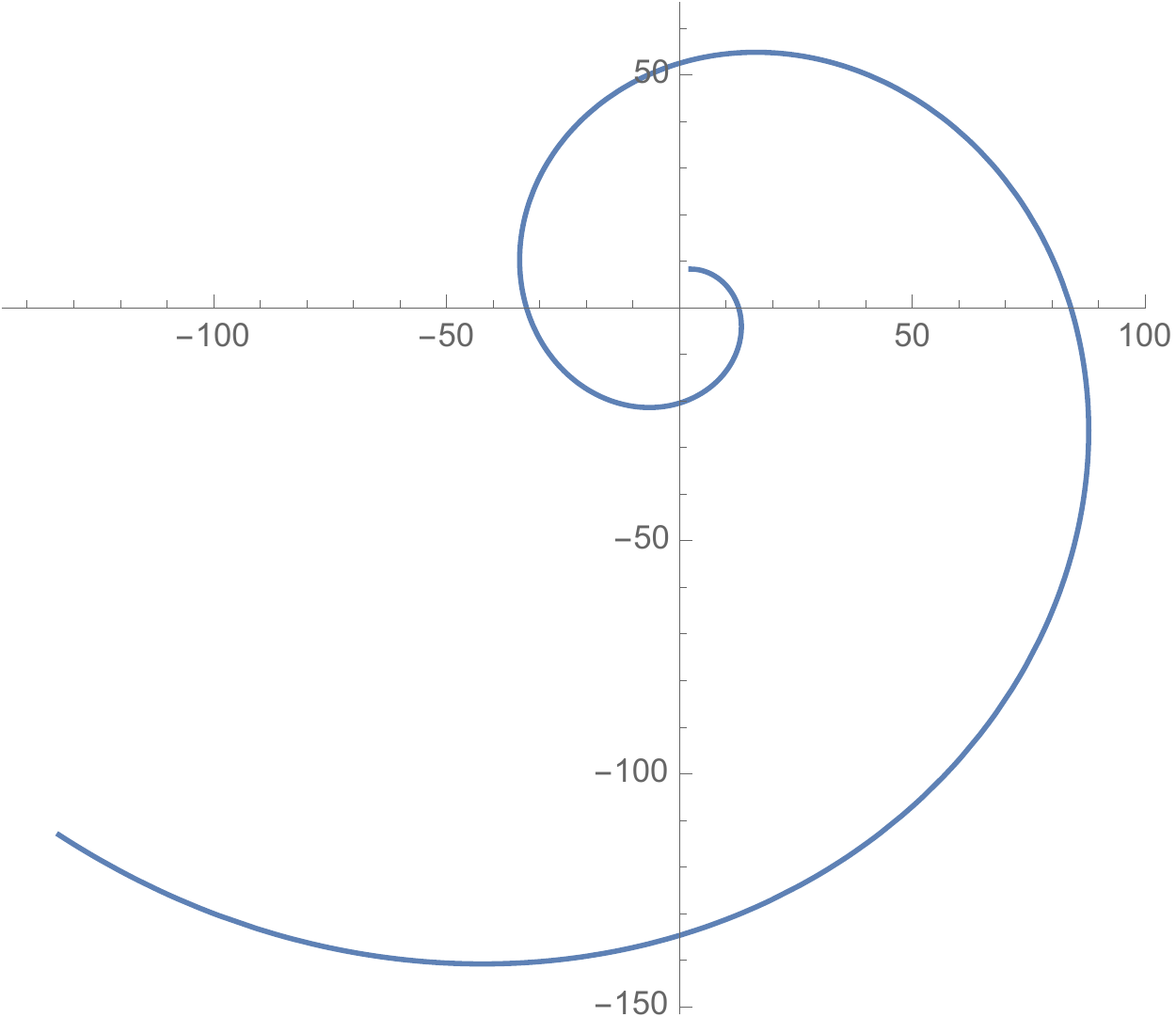}\quad \includegraphics[width=.4\textwidth]{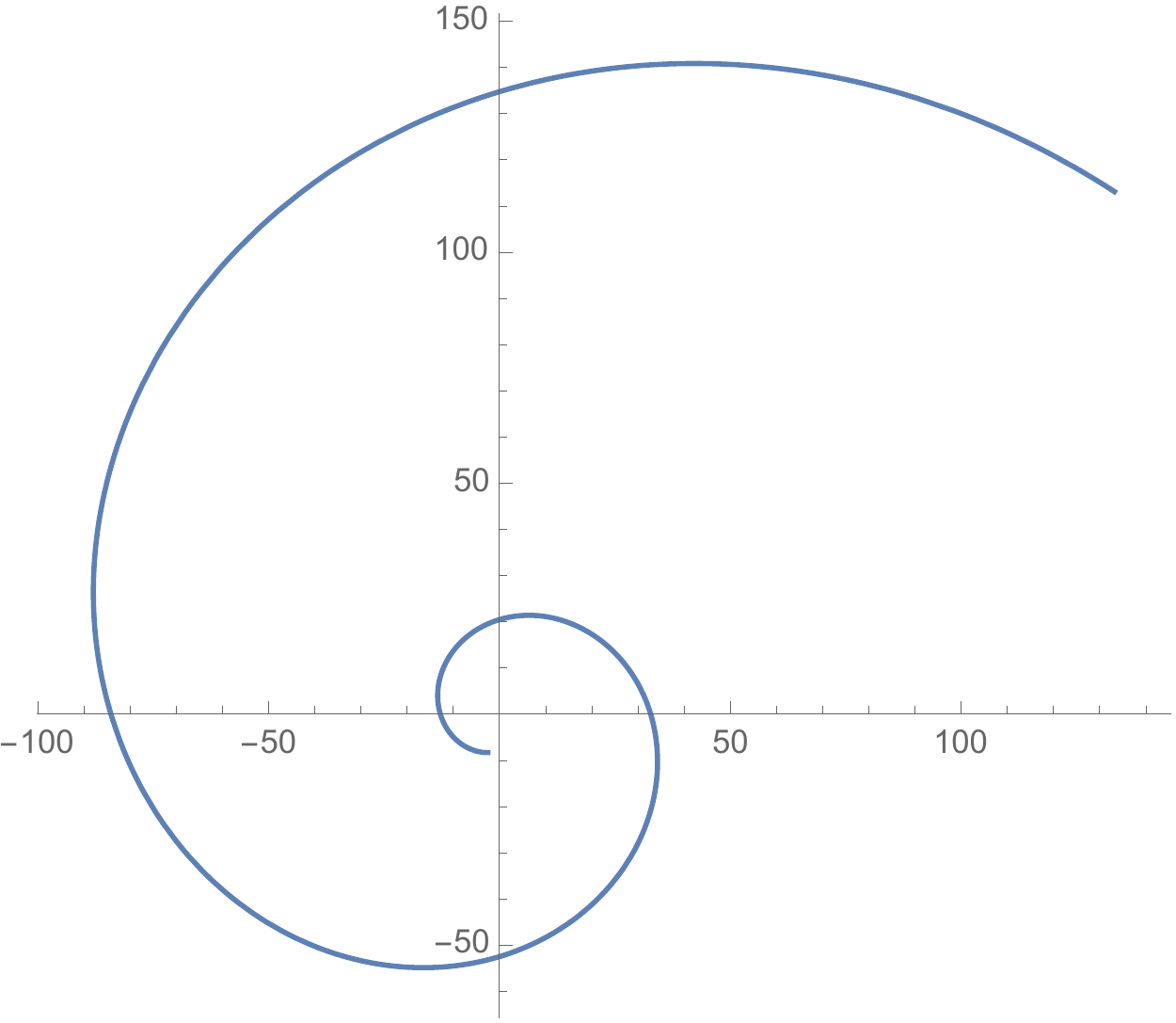}
\end{center}
\caption{Constant-speed ramps whose TreadmillSleds are half-lines of  Figure \ref{f-line}, left.  Here $v=2$, $\mu=0.1$. Left: parametrization \eqref{v2} for $\mathbf{a}$. Right: parametrization \eqref{v2} for $-\mathbf{a}$.}\label{f-v2}
\end{figure}

  \begin{figure}[hbtp]
\begin{center}\includegraphics[width=.4\textwidth]{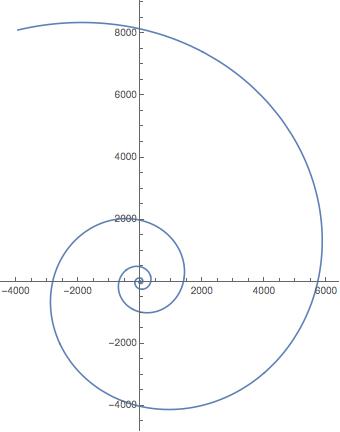}\quad \includegraphics[width=.4\textwidth]{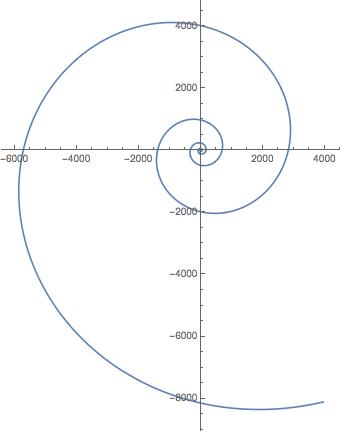}
\end{center}
\caption{Constant-speed ramps $\alpha_{ls}$ whose TreadmillSleds $\gamma_{ls}$ are the half-lines of  Figure \ref{f-line}, right.  Here $v=0.5$ and $\mu=0.3$. Left: parametrization \eqref{v2} for $\mathbf{a}$. Right: parametrization \eqref{v2} for $-\mathbf{a}$.}\label{f-v3}
\end{figure}

\begin{theorem} \label{t-41}
Let $v\not=1$ and $\mu>0$. For  the inverse central harmonic oscillator,  there are two constant-speed ramps $\alpha_{ls}$  that are logarithmic spirals, whose TreadmillSled curves $\gamma_{ls}$ are half-lines through the origin $O$.
\end{theorem}
 
In the rest of this section,   we address the problem of finding  the constant-speed ramps whose TreadmillSleds are not half-lines.  We use polar coordinates, so let 
$$\xi_1(t)=r(t)\cos(\varphi(t)),\quad \xi_2(t)= - r(t)\sin(\varphi(t)).$$
Then the differential equation system \eqref{eqs6} is now
\begin{eqnarray*}
&&-r' \cos  \varphi+r \varphi' \sin  \varphi-\frac{1}{2} r^2 \left(\sin (2 \varphi)+\mu  \left(\cos (2 \varphi)+2 v^2-1\right)\right)=0\\
&&r' \sin  \varphi+r \varphi' \cos  \varphi+r^2 \cos  \varphi (\mu  \sin  \varphi-\cos  \varphi)=0.
\end{eqnarray*}
Combining both equations, we deduce 
\begin{eqnarray*}
&&r' + \mu  v^2 r^2 \cos  \varphi=0\\
&&\varphi' - r \left(\cos  \varphi+\mu  \left(v^2-1\right) \sin  \varphi\right)=0.
\end{eqnarray*}
Let $\varphi=\varphi(r)$. Then
$$\varphi'(r)=\frac{\varphi'(t)}{r'(t)}=-\frac{r \left(\cos  \varphi+\mu  \left(v^2-1\right) \sin  \varphi\right)}{\mu  v^2 r^2 \cos  \varphi}=\frac{-1+\mu(1-v^2)\tan\varphi}{\mu  v^2 r},$$
or equivalently,  
$$\frac{\varphi'(r)}{-1+\mu(1-v^2)\tan\varphi}dr=\frac{1}{\mu v^2 }\frac{dr}{r}.$$
Firstly, we find the solution that satisfy that $r(0)=1$ and $\varphi(0)=0$, and later on, we     explain the reason why it is enough to only consider  this solution. With the change $u=\tan\varphi$, we find
\begin{equation}\label{p1}
\int\frac{du}{(1+u^2)(-1+\mu(1-v^2)u)}=\frac{1}{\mu v^2}\log(r).
\end{equation}
The first integral writes as 
\begin{equation}\label{p2}
\frac{1}{ A}\int\frac{B u-1}{u^2+1 }du-\frac{B^2}{A}\int\frac{1}{  B u+1}du
\end{equation}
 where
\begin{equation}\label{ab}
A=1+\mu ^2(v^2-1)^2, \quad B=\mu (v^2-1).
\end{equation}
By integrating \eqref{p2}, Equation \eqref{p1} is now
$$-\frac{\varphi}{A}+\frac{B}{A}\log(\sec\varphi)-\frac{B}{A}\log(1+B\tan\varphi)=\frac{1}{\mu v^2}\log(r)$$
or equivalently,
$$-\frac{\varphi}{A}+\frac{B}{A}\log\left(\frac{\sec\varphi}{1+B\tan\varphi}\right)=\frac{1}{\mu v^2}\log(r)$$

Thus the TreadmillSled $\gamma=\gamma(\varphi)$ of the constant-speed ramp is given in polar coordinates as
\begin{equation}\label{rf}
\begin{split}
r(\varphi)&=\exp\left(\mu v^2\left(-\frac{\varphi}{A}+\frac{B}{A}\log\left(\frac{\sec^2\varphi}{1+B\tan\varphi}\right)\right)\right)\\
&=e^{-\mu v^2\varphi/A}\left(\frac{\sec\varphi}{1+B\tan\varphi}\right)^{\mu v^2 B/A}=e^{-\mu v^2\varphi/A}\left(\frac{1}{\cos\varphi+B\sin\varphi}\right)^{\mu v^2 B/A}
\end{split}
\end{equation}

The domain of this solution is the interval

$$I=(-\arctan(1/B),-\arctan(1/B)+\pi)$$

Before we continue let us point out some important remarks.

\begin{remark} \label{rem1} A direct computation shows that if $(\phi_1(t),\phi_1(t))$ is a solution of the system \eqref{eqs6}, then, for any $\lambda$, $t\longrightarrow \lambda (\phi_1(\lambda t),\phi_2(\lambda t))$ is also a solution.
\end{remark}

\begin{remark} \label{rem2} If $v>1$ then $B>0$ and therefore the expression $\left(\frac{1}{\cos\varphi+B\sin\varphi}\right)^{\mu v^2 B/A}$ converges to infinity when $\varphi$ approaches the boundary values of the interval $I$. Recall that $\cos\varphi+B\sin\varphi$ is always positive on $I$ and it approaches zero at the boundaries of $I$. Therefore $r(\varphi)$ goes to $\infty$ as $\varphi$ goes to the boundary values of $I$. We also point out that the boundary values of $\varphi$ agree with the polar coordinate angles of the two semi-line solutions of the system given by the vector $ \mathbf{a}$ provided by the Equations \eqref{eqs63}. Therefore, if $v>1$, the trajectory given by Equation \eqref{rf} is a graph over the line spanned by $\mathbf{a}$ and it has this line as asymptote. By Remark \ref{rem1} we have that every trajectory of the system \eqref{eqs6} that is not one of the two semi-lines is  a dilation of the solution described in polar coordinate by Equation \eqref{rf}. See Figure \ref{f-phase2}. \end{remark}

\begin{remark} \label{rem3} If $v<1$ then $B<0$ and therefore the expression $\left(\frac{1}{\cos\varphi+B\sin\varphi}\right)^{\mu v^2 B/A}$ converges to zero when $\varphi$ approaches the boundary values of the interval $I$.  Therefore $r(\varphi)$ goes to zero as $\varphi$ goes to the boundary values of $I$. Recall that the boundary values of $\varphi$ agree with the polar coordinate angle of the two semi-line solutions of the system. Therefore, for  $v<1$, if we add the origin to the trajectory given by Equation \eqref{rf} we obtain a closed curve that is topologically a circle. By Remark \ref{rem1}, we have that every trajectory of the system \eqref{eqs6} that is not one of the two semi-lines is  a dilation of the solution described in polar coordinate by Equation \eqref{rf}. Even though the non semi-lines solution are topologically a circle with a point removed, geometrically they look more like a semi-circle connected with a diameter segment with one point removed. See Figure \ref{f-phase2}.

 \end{remark}

We now compute the constant-speed ramp in terms of the variable $\varphi$ by using  Proposition \ref{pr-t}. The function $f(t)$ is 
$$f(\varphi)=-\frac{\xi_2'(\varphi)}{\xi_1(\varphi)}=\frac{1-\mu\tan\varphi}{1+B\tan\varphi}.$$
On the other hand,
$$\xi_2 f-\xi_1'=e^{-\mu v^2\varphi/A}\left(\frac{\sec\varphi}{1+B\tan\varphi}\right)^{\frac{2B^2+\mu B+1}{A}}$$

Clearly this function is positive on the interval $I$ defined above. The function $G(t)$ is 
$$G(\varphi)=\int f(\varphi)d\varphi=\frac{-\mu v^2\log \left(B \sin (\varphi )+\cos (\varphi )\right)+(\mu B-1) \varphi }{1+B^2}.$$
Finally \eqref{tsa} yields
\begin{eqnarray*}
\alpha(\varphi)&=&-r(\varphi)\left(\begin{array}{ll} \cos(G(\varphi))&-\sin(G(\varphi))\\ \sin(G(\varphi))&\cos(G(\varphi))\end{array}\right)\left(\begin{array}{c} \cos\varphi\\ -\sin\varphi\end{array}\right)\\
&=&r(\varphi)\left(\begin{array}{l}
-\cos \left(\dfrac{\mu v^2 \left(B\varphi  -\log (B \sin (\varphi )+\cos (\varphi ))\right)}{B^2+1}\right)\\
\sin \left(\dfrac{\mu v^2 \left(B\varphi  -\log (B \sin (\varphi )+\cos (\varphi ))\right)}{B^2+1}\right)\end{array}\right)
\end{eqnarray*}

We summarize the previous arguments in the following theorem.
 
\begin{theorem} \label{t-42}
Let $v\not=1$ and $\mu>0$. For  the inverse central harmonic oscillator, the constant-speed ramps      whose TreadmillSled curves are not straight-lines are parametrized, up to a rotation and a dilation,  in polar coordinates by
\begin{equation}\label{eq-42}
\alpha(\varphi)=r(\varphi)\left(-\cos \Theta(\varphi),\sin \Theta(\varphi)\right),
\end{equation}
where  the radius function $r(\varphi)$  is defined    in \eqref{rf} and  
$$\Theta(\varphi)=\dfrac{\mu v^2 \left(B\varphi  -\log (B \sin (\varphi )+\cos (\varphi ))\right)}{B^2+1}.$$
\end{theorem}

\begin{remark} If we make $v=1$ in the expression \eqref{eq-42}, we obtain  
$$\alpha(\varphi)=-e^{-\mu  \varphi }\left( \cos (\mu  \log (\cos (\varphi ))),  \sin (\mu  \log (\cos (\varphi )))\right), $$

which is one of the ramps described in Equation \eqref{alphaforveq1} of Theorem \ref{tv1}.
\end{remark}

We now present some pictures of constant-speed ramps that are not spirals. The figures will be implemented in Mathematica 12.0 (\cite{wo}). In Figure \ref{f-v4} we show a case with $v>1$, where both, the ThreadmillsSled curve as the constant-speed ramp appear. 

  \begin{figure}[hbtp]
\begin{center}\includegraphics[width=.2\textwidth]{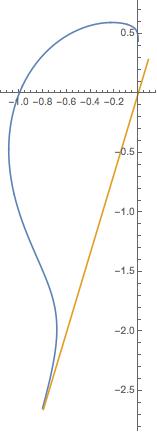}\quad \includegraphics[width=.4\textwidth]{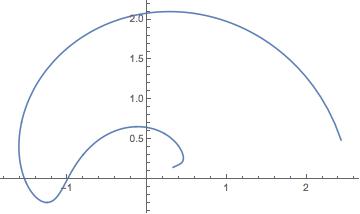}
\end{center}
\caption{Constant-speed ramps whose TreadmillSleds are not half-lines.  Here $v=2$ and $\mu=0.1$. Left:  the ThreadmillSled  which is asymptotic to the  line of vector $\mathbf{a}$.   Right: the constant-speed ramp.}\label{f-v4}
\end{figure}

The constant-speed ramp for the case $v<1$ appears in Figure \ref{f-v05}, together with its  ThreadmillsSled curve. Finally,   Figure \ref{f-bueno} displays both constant-speed ramps.

  \begin{figure}[hbtp]
\begin{center}\includegraphics[width=.2\textwidth]{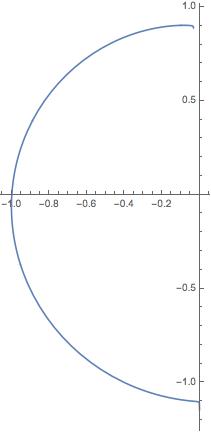}\quad \includegraphics[width=.4\textwidth]{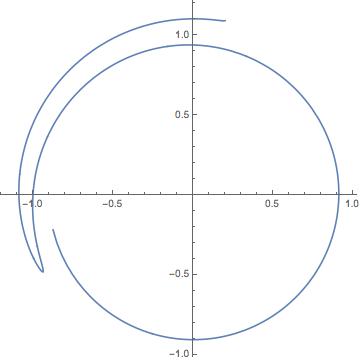}
\end{center}
\caption{Constant-speed ramps whose TreadmillSleds are not half-lines.  Here $v=0.8$ and $\mu=0.3$. Left: the ThreadmillSled.   Right: the constant-speed ramp.}\label{f-v05}
\end{figure}

  \begin{figure}[hbtp]
\begin{center}\includegraphics[width=.3\textwidth]{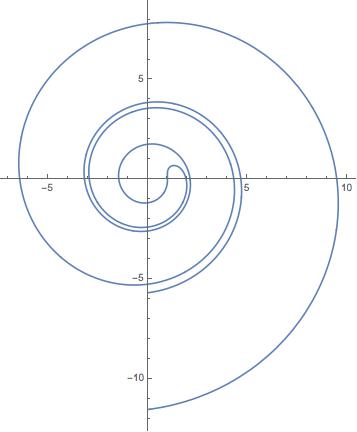}\quad \includegraphics[width=.4\textwidth]{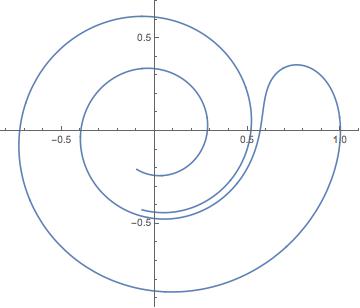}
\end{center}
\caption{Constant-speed ramps that are not spirals. Left: $v=2$ and $\mu=0.1$. Right: $v=0.8$ and $\mu=0.3$. }\label{f-bueno}
\end{figure}

From the analysis of the phase portrait in Figure \ref{f-phase2}, we conclude the following qualitative properties.

\begin{corollary}\label{c-43}
If $\alpha$ is a constant-speed ramp described in Equation \eqref{eq-42}, then  $\alpha$   converges to one of the spirals of Theorem \ref{t-41} with the same value of $v$ and $\mu$. Moreover, if   $v>1$ (resp. $v<1$),  $\alpha$ is not bounded (resp. bounded) curve.
\end{corollary}

\begin{proof} By Remarks \ref{rem2} and \ref{rem3}, we know that the TreadmillSled $\gamma$ of $\alpha$ converges asymptotically to one of the two half-lines $\gamma_{ls}$ of Theorem \ref{t-41}, which are TreadmillSled of logarithmic spirals $\alpha_{ls}$. Thus $\alpha$ converges to one of these spirals $\alpha_{ls}$. See Figure \ref{f-bueno}, left.

On the other hand, if $v>1$,  the TreadmillSleds curve $\gamma$ is unbounded, and hence, $\alpha$ are also unbounded. In case $v<1$, we see that the trajectories of the TreadmillSled in Figure \ref{f-phase2} are bounded and are topologically circles with one point removed. Near to the origin, they are tangent to the two half-lines $\gamma_{ls}$. Thus the constant-speed ramps $\alpha$ are also bounded and converge to the spirals $\alpha_{ls}$. See Figure \ref{f-bueno}, right.
\end{proof}

\bibliographystyle{amsplain}

\end{document}